\documentclass{amsart}
\usepackage{xypic}

\usepackage{
amssymb,
amsmath,
amsthm,
graphicx,
amscd,
multind,
eufrak,
}
\theoremstyle{plain}
\newtheorem{thm}{Theorem}[section]
\newtheorem{lm}[thm]{Lemma}
\newtheorem{prop}[thm]{Proposition}
\theoremstyle{definition}
\newtheorem{de}[thm]{Definition}
\newtheorem{ex}[thm]{Example}
\newtheorem{re}[thm]{Remark}
\newcommand{\RR}{{\mathbb R}}
\newcommand{\NN}{{\mathbb N}}
\renewcommand{\AA}{{\mathbb A}}
\newcommand{\im}{\operatorname{im}}
{\begin{figure} \begin{center}}%
{\end{center} \end{figure}}

\def\sgn{{\rm sgn}\,}

\newcommand{\Spec}{\operatorname{Spec}}
\newcommand{\Sym}{\operatorname{Sym}}
\newcommand{\id}{\operatorname{id}}

\newcommand{\SL}{\operatorname{SL}\nolimits}

\newcommand{\Mmn}[1]{\mathrm{M}_{#1}}
\newcommand{\Mn}[1][n]{\mathrm{M}_{#1}}
\newcommand{\SMn}[1][n]{\mathrm{SM}_{#1}}
\newcommand{\OMn}[1][n]{\mathrm{OM}_{#1}}
\newcommand{\SOMn}[1][n]{\mathrm{SOM}_{#1}}
\newcommand{\Mnk}[2][k]{\mathrm{M}_{#2}^{\leq #1}}
\newcommand{\SMnk}[2][k]{\mathrm{SM}_{#2}^{\leq #1}}
\newcommand{\OMnk}[2][k]{\mathrm{OM}_{#2}^{\leq #1}}
\newcommand{\SOMnk}[2][k]{\mathrm{SOM}_{#2}^{\leq #1}}
\newcommand{\Snk}[2][k]{\mathrm{S}_{#2}^{(#1)}}
\newcommand{\Vnk}[2][k]{\mathrm{V}_{#2}^{(#1)}}
\newcommand{\tOMk}[1][k]{\tilde{\mathrm{OM}}_{\infty}^{\leq #1}}
\newcommand{\tVk}[1][k]{\tilde{\mathrm{V}}_{\infty}^{(#1)}}
\newcommand{\tA}{\tilde{A}_\infty}

\begin{document}
\title{Finiteness for the $k$-factor model\\ and chirality varieties}
\author[J.~Draisma]{Jan Draisma}
\address[Jan Draisma]{
Department of Mathematics and Computer Science\\
Technische Universiteit Eindhoven\\
P.O. Box 513, 5600 MB Eindhoven, Netherlands\\
and Centrum voor Wiskunde en Informatica, Amsterdam, The
Netherlands}
\thanks{The author is supported by DIAMANT, an NWO
mathematics cluster.}
\email{j.draisma@tue.nl}
\maketitle

\begin{abstract}
This paper deals with two families of algebraic varieties arising from
applications. First, the {\em $k$-factor model} in statistics, consisting
of $n \times n$ covariance matrices of $n$ observed Gaussian variables
that are pairwise independent given $k$ hidden Gaussian variables. Second,
{\em chirality varieties} inspired by applications in chemistry. A
point in such a chirality variety records chirality measurements of all
$k$-subsets among an $n$-set of ligands. Both classes of varieties are
given by a parameterisation, while for applications having polynomial
equations would be desirable. For instance, such equations could be used
to test whether a given point lies in the variety. We prove that in a
precise sense, which is different for the two classes of varieties, these
equations are finitely characterisable when $k$ is fixed and $n$ grows.
\end{abstract}

\section{Results}

\subsection*{The $k$-factor model}
Factor analysis addresses the problem of testing whether $n$ observed
random variables are conditionally independent given $k$ hidden variables,
called the factors. In the case where the joint distribution of all $n+k$
variables is multivariate Gaussian, the parameter space $F_{n,k}$ for
the $k$-factor model is the set of $n \times n$-covariance matrices of
the form $\Sigma + \Gamma$ where $\Sigma$ is diagonal positive definite
and $\Gamma$ is positive semidefinite of rank at most $k$. An algebraic
approach to factor analysis seeks to determine all polynomial relations
among the matrix entries in $F_{n,k}$; these relations are called model
invariants \cite{Drton07}.

Clearly, any principal $m \times m$-submatrix of a matrix in $F_{n,k}$
lies in $F_{m,k}$. An important question of theoretical interest is
whether, for fixed $k$, there exists an $m$ such that for $n \geq m$
the model $F_{k,n}$ is completely characterised by the fact that
each principal $m \times m$-submatrix lies in $F_{k,m}$. For $k=2$
this question was settled in the affirmative very recently; $m=6$
suffices \cite{Drton08}.  We prove the corresponding statement for
the Zariski closure of the model, i.e., for the set of all real (or
complex) $n \times n$-matrices satisfying all model invariants. Apart from
polynomial equalities the definition of the model $F_{k,n}$ also involves
inequalities, which our approach does not take into consideration.

Our theorem to this effect needs the following notation. Let $K$ be
a field; all varieties and schemes will be defined over $K$. If $X$
is a scheme over $K$ and $S$ is a $K$-algebra (commutative with $1$),
then we write $X(S)$ the set of $S$-rational points of $X$. Our schemes
will be affine, but not necessarily of finite type. So $X=\Spec R$
for some $K$-algebra $R$ and $S$-rational points are the $K$-algebra
homomorphisms $R \rightarrow S$.

For a natural number $n$ we write $[n]$ for the set $\{1,\ldots,n\}$
and $\Mn,\SMn$ for the affine spaces over $K$ of $n \times n$-matrices
and of symmetric $n \times n$-matrices, respectively. We also write
$\OMn$ for the affine space of {\em off-diagonal $n \times n$-matrices}.
This is the space $\AA^{n^2-n}$ where we think of the coordinates as the
off-diagonal entries $y_{ij},\ i \neq j$ of an $n \times n$-matrix, so
that the notion of principal submatrix of an off-diagonal matrix has an
obvious meaning. Similarly we write $\SOMn$ for the space of symmetric
off-diagonal $n \times n$-matrices. There are natural projections
$\Mn \rightarrow \OMn$ and $\SMn \rightarrow \SOMn$. Given a second
natural number $k$ we write $\Mnk{n} \subseteq \Mn$ for the subvariety
of matrices of rank at most $k$ and $\SMnk{n} \subseteq \SMn$ for the
subvariety of symmetric matrices of rank at most $k$. Our first finiteness
result concerns the varieties $\OMnk{n}$ and $\SOMnk{n}$, which are the
scheme-theoretic images of $\Mnk{n}$ and of $\SMnk{n}$, respectively. In
concrete terms, the ideal of $\OMnk{n}$ is the intersection of the ideal
of $\Mnk{n}$ with the polynomial algebra in the off-diagonal matrix
entries, and similarly for $\SOMnk{n}$. It seems rather hard to determine
these ideals explicitly; for fixed $k$ and $n$ they can in principle be
computed using Gr\"obner basis techniques \cite{Drton07}.

\begin{ex}
For $k=2$ and $n=5$ the variety $\SOMnk[2]{5}$ is a
hypersurface in $\SOMn[5]$ with equation 
\[ \frac{1}{10} \sum_{\sigma \in \Sym(5)} 
	\sgn(\sigma) \sigma(y_{12}y_{23}y_{34}y_{45}y_{51})=0 ,
\]
where $y_{ij}=y_{ji}$ is the $(i,j)$-matrix entry and $\Sym(5)$
acts by simultaneously permuting rows and columns. The factor $1/10$
comes from the dihedral group stabilising the $5$-cycle, so this
equation has $12$ terms. This equation is called the {\em pentad} in
\cite{Drton07}. Experiments there show that for $n$ up to $9$ pentads
and off-diagonal $3 \times 3$-minors generate the ideal of $\SOMnk[2]{n}$.
\end{ex}

For any subset $I \subseteq [n]$ of size $m$ and any matrix $y$ in $\Mn$
we write $y[I] \in \Mn[m]$ for the principal submatrix of $y$ with rows
and columns labelled by $I$; this notation is also used for off-diagonal
matrices. If $y \in \Mnk{n}(K)$, then also $y[I] \in \Mnk{m}(K)$.
Conversely, $y \in \Mn(K)$ lies in $\Mnk{n}(K)$ if and only if all
its $(k+1) \times (k+1)$-minors vanish. This implies that if $n \geq
2(k+1)$, then $y \in \Mn(K)$ lies in $\Mnk{n}(K)$ if and only if $y[I]
\in \Mnk{2(k+1)}(K)$ for all $I \subseteq [n]$ of size $2(k+1)$. Moreover,
this statement holds scheme-theoretically, as well: the ideal of $\Mnk{n}$
is generated by the pullbacks of the ideal of $\Mnk{2(k+1)}$ under the
morphisms $y \mapsto y[I]$; this is just a restatement of the well-known
fact that the $(k+1)\times(k+1)$-minors generate the ideal of $\Mnk{n}$
\cite{Bruns88}. Note that we need $2(k+1)$ here, rather than for instance
$(k+1)$, because we are only taking {\em principal} submatrices.

\begin{thm}[Set-theoretic finiteness for the $k$-factor model]
\label{thm:KFactorModel}
There exists a natural number $N_0$, depending only on $k$,
such that for all $n \geq N_0$ we have
\[ \OMnk{n}(K)=\{y \in \Mn(K) \mid y[I] \in \OMnk{N_0}(K)
	\text{ for all } I \subseteq [n] \text{ of size }
	N_0 \}. \]
Similarly, there exists a natural number $N_1$, depending
only on $k$, 
such that for all $n \geq N_1$ we have 
\[ \SOMnk{n}(K)=\{y \in \SMn(K) \mid y[I] \in \SOMnk{N_1}(K)
	\text{ for all } I \subseteq [n] \text{ of size }
	N_1 \}. \]
\end{thm}

This theorem settles the ``radical'' part of \cite[Question 29]{Drton07}.
To relate this theorem to the $k$-factor model take $K=\RR$. Since
the diagonal entries in $F_{n,k}$ are ``free parameters'', all model
invariants are generated by those involving only the off-diagonal
entries. Hence a matrix lies in the Zariski closure of $F_{n,k}$ if
and only if its image in $\SOMn(\RR)$ lies in $\SOMnk{n}(\RR)$. Note
that the theorem does not claim the stronger finiteness property in
\cite[Question 29]{Drton07}, that the entire ideal of $\SOMnk{n}$ is
generated by the pull-back of the ideal of $\SOMnk{N_1}$ under taking
principal submatrices. Although we expect this to be true, our methods
do not suffice to prove this result.

\subsection*{Chirality varieties}

Our second finiteness result concerns another family of algebraic
varieties, motivated by applications in chemistry such as the following
\cite{Ruch67}. One imagines four distinct ligands in the vertices of a
regular tetrahedron $T$, which bond to an atom in the centre of $T$,
and one is interested in some measurable property of the resulting
structure which is invariant under orientation preserving symmetries
of $T$ but not under reflection. An example of such a property is
optical activity. One assumes that the ligands can be characterised
by scalars $x_1,\ldots,x_4$ and that the property is captured by a
scalar valued function F of those four variables. The smallest-degree
case is that where $F(x_1,\ldots,x_4):=\prod_{i<j}(x_i-x_j)$. This
function is called a {\em chirality product}; it changes its sign under
reflections. Next one assumes that one has $n$ ligands, and wants to know
the polynomial relations among the values $F(x_{i_1},\ldots,x_{i_4})$
as $\{i_1,\ldots,i_4\}$ runs over all $4$-subsets of $[n]$. Again one
hopes that these relations are characterised by finitely many types;
and this is exactly what we shall prove.

More precisely and generally, fix a natural number $k$ and for all $n
\geq k$ consider the affine space $\Snk{n}$ whose coordinates $y_J$ are
parameterised by all $k$-subsets of $[n]$. In the example above $k$ equals
$4$. Consider the morphism $\AA^n \rightarrow \Snk{n}$ that sends $x$ to
the point $y \in \Snk{n}$ whose coordinate $y_J$ equals the Vandermonde
determinant $\prod_{i,j \in J: i<j}(x_i-x_j)$. The scheme-theoretic image
of this map is denoted $\Vnk{n}$. We call $\Vnk{n}$ a {\em chirality
variety}; its ideal is the set of all relations between the Vandermonde
determinants, or outcomes of chirality measurements for all $k$-subsets of
$n$ ligands. For any subset $I$ of $[n]$ of size $m \geq k$ we have a
natural morphism $\pi_I:\Snk{n} \rightarrow \Snk{m},\ y \mapsto y[I]$,
which forgets the coordinates corresponding to $k$-subsets $J \not
\subseteq [m]$. By construction, this map sends $\Vnk{n}$ into $\Vnk{m}$.

\begin{thm}[Scheme-theoretic finiteness for chirality
varieties in characteristic zero]
\label{thm:VandermondeScheme}
In the (chemically relevant) case where the characteristic of $K$ is zero
there exists a natural number $N_2$, depending only on $k$, such that for
all $n \geq N_2$ the scheme $\Vnk{n}$ is the scheme-theoretic intersection
of the pre-images $\pi_I^{-1} \Vnk{N_2}$ over all $N_2$-subsets $I$
of $[n]$.
\end{thm}

In more concrete terms, the ideal of $\Vnk{n}$ is generated by the
pull-backs of $\Vnk{N_2}$ under all projections $\pi_I$. The condition
on the characteristic of $K$ comes from the fact that our proof uses
the existence of a Reynolds operator.

\subsection*{Preview}
The remainder of this paper is organised as follows. In Section
\ref{sec:FinitenessProblems} we cast our two main results in
a common framework, in which a sequence of schemes with group
actions is replaced by its limit, which is a scheme of infinite
type. Section \ref{sec:TopGNoetherianity} introduces and develops
the new notion of $G$-Noetherianity for topological spaces. Lemma
\ref{lm:FibreBundle} is particularly important for our proof of Theorem
\ref{thm:KFactorModel}. Section \ref{sec:SchemeGNoetherianity} introduces
and develops the analoguous notion for rings (or schemes). Theorem
\ref{thm:AHS} by Aschenbrenner, Hillar, and Sullivant is used extensively
in all proofs, while Proposition \ref{prop:Reynolds} shows how to use
the Reynolds operator to prove scheme-theoretic finiteness for certain
schemes, including the chirality varieties in characteristic $0$.

\section*{Acknowledgments}
I thank Jochen Kuttler for pointing out that the Reynolds operator can
be used to translate my earlier set-theoretic finiteness result for
chirality varieties into the present scheme-theoretic result. I also
thank Seth Sullivant for numerous discussions on the subject, and for
introducing me to the field of algebraic statistics.  Finally, I thank
Mathias Drton for sharing his results in \cite{Drton08} at an early stage.

\section{Finiteness problems for chains of schemes}
\label{sec:FinitenessProblems}

Our results above fit into the following set-up, which is similar to
that of \cite[Section 4]{Aschenbrenner07}.  First, we are given an
infinite sequence
\[
\xymatrix{
A_1 \ar@<2pt>@{^(->}[r]^{\tau_{12}} &
A_2 \ar@<2pt>[l]^{\pi_{21}} \ar@<2pt>@{^(->}[r]^{\tau_{23}} &
A_3 \ar@<2pt>[l]^{\pi_{32}} \ar@<2pt>@{^(->}[r]^{\tau_{34}} &
\ldots \ar@<2pt>[l]^{\pi_{43}}
}
\]
where each $A_n=\Spec(T_n)$ is an affine scheme over $K$, $\tau_{n,n+1}$
is a closed embedding, and $\pi_{n+1,n}$ is a morphism satisfying
$\pi_{n+1,n} \tau_{n,n+1}=\id_{A_n}$. For $m \leq n$ we define
$\tau_{m,n}:=\tau_{n-1,n} \cdots \tau_{m,m+1}$ and $\pi_{n,m}:=\pi_{m+1,m}
\cdots \pi_{n,n-1}$. 
Second, we are given a sequence
\[
\xymatrix{
Y_1 \ar@<2pt>@{^(->}[r]^{\tau_{12}} &
Y_2 \ar@<2pt>[l]^{\pi_{21}} \ar@<2pt>@{^(->}[r]^{\tau_{23}} &
Y_3 \ar@<2pt>[l]^{\pi_{32}} \ar@<2pt>@{^(->}[r]^{\tau_{34}} &
\ldots \ar@<2pt>[l]^{\pi_{43}}
}
\]
where each $Y_n$ is a closed subscheme of $A_n$ with ideal $I_n \subseteq
A_n$, and we require that $\pi_{n+1,n}$ maps $Y_{n+1}$ into $Y_n$ and
that $\tau_{n,n+1}$ maps $Y_n$ into $Y_{n+1}$.
Third, we are given a sequence of groups
\[ G_1 \subseteq G_2 \subseteq \ldots \]
with for each $n$ an action of $G_n$ on $A_n$ by automorphisms stabilising
$Y_n$, and such that $\pi_{n+1,n}$ is $G_n$-equivariant. We do {\em not}
require that $\tau_{n,n+1}$ be $G_n$-equivariant.  We make the following
somewhat technical assumption: for any triple $q \geq n \geq m$ and any
$g \in G_q$ there exist a $p \leq m$, $g' \in G_n,$ and $g'' \in G_m$
such that
\begin{equation} \label{eq:Commute} \tag{*}
\pi_{q,m}g\tau_{n,q}=g''\tau_{p,m}\pi_{n,p}g'.
\end{equation}
Now let $A_\infty:=\lim_{\leftarrow n} A_n$ be the projective limit
of the $A_n$, i.e., the affine scheme corresponding to the algebra
$T_\infty:=\bigcup_n T_n$, where for $m \leq n$ the algebra $T_m$ is
identified with a subalgebra of $T_n$ by means of $\pi_{n,m}^*$. Since
the latter map is $G_m$-equivariant, the union $G$ of all $G_n$ acts
naturally on $T_\infty$ and hence on $A_\infty$. Define $Y_\infty
\subseteq A_\infty$ similarly; then $G$ stabilises $Y_\infty$. We write
$\tau_{n,\infty}$ and $\pi_{\infty,n}$ for the natural embedding $A_n
\rightarrow A_\infty$ and the natural projection $A_\infty \rightarrow
A_n$, respectively.

\begin{lm} \label{lm:InftyChain}
Let $m$ be a natural number. For any $K$-algebra $S$ the following two
statements are equivalent:
\begin{enumerate}
\item for all $n \geq m$ the set $Y_n(S)$ consists of all $y \in A_n(S)$
for which $\pi_{n,m} G_n y \subseteq Y_m(S)$, and
\label{it:chain}
\item the set $Y_\infty(S)$ consists of all $y \in A_\infty(S)$ such
that $\pi_{\infty,m} G y \subseteq Y_m(S)$.
\label{it:infty}
\end{enumerate}
\end{lm}

\begin{proof}
For the implication \eqref{it:chain} $\Rightarrow$ \eqref{it:infty}
let $y \in A_\infty(S)$ have the property that $\pi_{\infty,m}
Gy \subseteq Y_m(S)$. By definition, the element $y$ lies in
$Y_\infty(S)$ if and only if $\pi_{\infty,n} y$ lies in $Y_n(S)$
for all $n$. For $n \leq m$ this condition is fulfilled because
$\pi_{\infty,n}(y)=\pi_{m,n}\pi_{\infty,m}(y)$ and the right-hand side
lies in $Y_n(S)$ by the property of $y$ and the fact that $\pi_{m,n}$ maps
$Y_m(S)$ into $Y_n(S)$. Next consider $y':=\pi_{\infty,n} y \in A_n(S)$
for $n \geq m$. For $g \in G_n$ we have $\pi_{n,m} gy'=\pi_{\infty,m}g
y$, which lies in $Y_m(S)$ by assumption. Hence by \eqref{it:chain} $y'$
lies in $Y_n(S)$, as needed.

For the implication \eqref{it:infty} $\Rightarrow$ \eqref{it:chain}
let $y \in A_n(S)$ satisfy $\pi_{n,m} G_n y \subseteq Y_m(S)$. Let
$y':=\tau_{n,\infty} y$; we prove that $y'$ lies in $Y_\infty(S)$ by
showing that $\pi_{\infty,m} Gy' \subseteq Y_m(S)$. Indeed, let $g \in
G$, say $g \in G_q$. If $q \leq n$ then $g \in G_n$ and $\pi_{\infty,m}
gy'=\pi_{n,m} gy \in Y_m(S)$ by the condition on $y$. Hence suppose
that $q \geq n$ and consider the element $\pi_{\infty,m}g y'= \pi_{q,m}
g \tau_{n,q} y$. Now we invoke \eqref{eq:Commute} above to find a
$p \leq m$ and a $g' \in G_n$ such that the right-hand side equals
$g'' \tau_{p,m} \pi_{n,p} g'y$, which by the property of $y$ lies in
$Y_m(S)$. By \eqref{it:infty} we conclude that $y'$ does indeed lie in
$Y_\infty(S)$, hence $y$ lies in $Y_n(S)$.
\end{proof}

The condition that \eqref{it:chain} be true for all $K$-algebras $S$
is equivalent to the statement that for $n \geq m$ the ideal $I_n
\subseteq T_n$ of $Y_n$ be the smallest $G_n$-stable ideal containing
$I_m$. Similarly, the condition that \eqref{it:infty} be true for all
$K$-algebras $S$ is equivalent to the statement that the ideal $I_\infty$
be the smallest $G$-stable ideal of $T_\infty$ containing $I_m$.
Consequently, these two scheme-theoretic statements are also equivalent.

We now fit our main results into this set-up.

\subsection*{The $k$-factor model}

Here $A_n$ equals $\OMn$ or $\SOMn$, $Y_n$ equals $\OMnk{n}$ or
$\SOMnk{n}$, and $G_n=\Sym(n)$ acts by simultaneous row- and column
permutations. The map $\pi_{n+1,n}$ sends an (off-diagonal) $(n+1) \times
(n+1)$-matrix to its principal $n \times n$-submatrix in the upper
left corner, and $\tau_{n,n+1}$ augments an (off-diagonal) $n \times
n$-matrix with zero $(n+1)$-st row and column. That $\pi_{n+1,n}$ and
$\tau_{n,n+1}$ map $Y_{n+1}$ into $Y_n$ and vice versa follows from the
fact that they map $\Mnk{n+1}$ into $\Mnk{n}$ and vice versa. Finally,
condition \eqref{eq:Commute} is fulfilled since any $m \times m$-principal
submatrix of a $q \times q$-matrix obtained from an $n \times n$-matrix
$y$ by augmenting with zeros and applying an element of $\Sym(q)$ can
also be obtained by applying an element of $\Sym(n)$ to $y$, taking
a suitable principal $p \times p$-submatrix, augmenting with zeroes,
and reordering rows and columns with a permutation from $\Sym(m)$.

Our proof of Theorem \ref{thm:KFactorModel}, set-theoretic finiteness
of the $k$-factor model, will focus on the $K$-rational points of the
schemes $A_n=\OMn$ and $Y_n=\OMnk{n}$. More specifically, we shall prove
that there exist finitely many elements $f_1,\ldots,f_l \in T_\infty$ such
that $y \in A_\infty(K)$ lies in $Y_\infty(K)$ if and only if $f_i(gy)=0$
for $i=1,\ldots,l$ and $g \in G$. By Lemma \ref{lm:InftyChain} with $S=K$
this implies Theorem \ref{thm:KFactorModel}.

\subsection*{Chirality varieties} Here $A_n$ equals $\Snk{n}$, $Y_n$ equals
$\Vnk{n}$, and $G_n=\Sym(n)$ acts on the coordinates as follows: $gy_J$
equals $(-1)^a y_{gJ}$ where $a$ is the number of pairs $i<j$ in $J$
such that $gi>gj$; we call $a$ the number of inversions of $g$ on
$J$. Note that this action makes the parameterisation $y_J=\prod_{i,j
\in J: i<j} (x_i-x_j)$ of $\Vnk{n}$ $\Sym(n)$-equivariant.  The map
$\pi_{n+1,n}$ projects onto the coordinates $y_J$ with $J$ a subset
of $[n]$. That $\pi_{n+1,n}$ maps $Y_{n+1}$ into $Y_n$ is clear from
the parameterisation.  The map $\tau_{n,q}$ for $q \geq n$ is defined
by its dual as follows: for $J$ a $k$-subset of $[q]$ we set
\[ 
\tau_{n,q}^*: y_J \mapsto 
	\begin{cases}
	y_J & \text{ if } J \subseteq [n-1]\\
	0 & \text{ if } |J \setminus [n-1]| \geq 2, \text{ and }\\
	y_{J-\{j\}+\{n\}} & \text{ if } J \setminus
	[n-1]=\{j\}.
	\end{cases}
\]
This reflects the effect of taking $x_j=x_n$ for all $j>n$ in the
parameterisation, which shows that $\tau_{n,q}$ does indeed map $Y_n$
into $Y_q$. Note that $\tau_{n,q}=\tau_{q-1,q} \cdots \tau_{n,n+1}$,
as in our set-up. Now we need to verify condition \eqref{eq:Commute},
whose dual statement reads
\[ \tau_{n,q}^* g \pi_{q,m}^* =
	g'\pi_{n,p}^*\tau_{p,m}^*g'' \]
for suitable $p \leq m,g'\in \Sym(n),g''\in \Sym(m)$. Let $L$ be the
set of elements in $[m]$ that are mapped into $[n-1]$ by $g$. For any
$k$-subset $J$ of $[m]$ the map $\tau_{n,q}^* g
\pi_{q,m}^*$ sends 
\[ 
y_J \mapsto 
	\begin{cases}
	(-1)^a y_{gJ} & \text{ if } J \subseteq L\\
	0 & \text{ if } |J \setminus L| \geq 2, \text{ and}\\
	(-1)^a y_{gJ-\{gj\}+\{n\}} & \text{ if } J
	\setminus L=\{j\},
	\end{cases} 
\]
where $a$ is the number of inversions of $g$ on $J$.  We distinguish two
cases. First suppose that $L$ is all of $[m]$. Then the last two cases
do not occur, and we may take $g'':=1 \in \Sym(m)$, $p:=m$, and any $g'
\in \Sym(n)$ which agrees with $g$ on $[m]$.  Second, suppose that $L
\subsetneq [m]$. Set $p:=|L|+1 \leq m$ and choose $g'' \in \Sym(m)$ such
that $g''$ maps $L$ bijectively into $[p-1]$ and such that for all $i,j$
in $L$ we have $g''i<g''j$ if and only $gi<gj$. This ensures that the
number of inversions of $g''$ on any subset of $[m]$ containing at most
one element outside of $L$ is the same as the number of inversions of $g$
on that subset. Then $\pi_{n,p}^*\tau_{p,m}^* g''$ maps
\[ 
y_J \mapsto 
	\begin{cases}
	(-1)^a y_{g''J} & \text{ if } J \subseteq L\\
	0 & \text{ if } |J \setminus L| \geq 2, \text{ and }\\
	(-1)^a y_{g''J-\{g''j\}+\{p\}} & \text{ if } J \setminus L=\{j\}
	\end{cases},
\]
where $a$ is both the number of inversions of $g$ on $J$ and that of
$g''$ on $J$. Hence if we compose this with an element $g' \in \Sym(n)$
which is increasing on $[p]$ and satisfies $g'g''=g$ on $L$ and $g'p=n$,
then we are done.

For Theorem \ref{thm:VandermondeScheme}, scheme-theoretic finiteness of
chirality varieties in characteristic $0$, we shall prove that there
exist finitely many elements $f_1,\ldots,f_l$ in the ideal $I_\infty$
of $Y_\infty=\Vnk{\infty}$ whose $G$-orbits generate the ideal of
$Y_\infty$. By Lemma \ref{lm:InftyChain} and the remark following it
this implies Theorem \ref{thm:VandermondeScheme}.

\section{Topological $G$-Noetherianity}
\label{sec:TopGNoetherianity}

In this section we develop a purely topological notion that can be used
to prove set-theoretic finiteness results.

\begin{de}
Let $G$ be a group acting by homeomorphisms on a topological space $X$;
we shall call $X$ a $G$-space. Then $X$ is called {\em $G$-Noetherian}
if every chain $X_1 \supseteq X_2 \supseteq \ldots$ of closed $G$-stable
subsets of $X$ stabilises in the sense that there exists an $m$ such
that $X_n=X_m$ for all $n \geq m$.
\end{de}

\begin{lm} \label{lm:ClosedSubset}
Let $X$ be a $G$-space and $Y$ a $G$-stable closed subset of $X$. If $X$
is $G$-Noetherian, then so is $Y$.
\end{lm}

\begin{proof} 
Every descending chain of $G$-stable closed subsets in
$Y$ is also such a chain in $X$, hence stabilises.
\end{proof}

\begin{lm} \label{lm:Surjection}
Let $X$ and $Y$ be $G$-spaces. If $Y$ is $G$-Noetherian and if there
exists a surjective $G$-equivariant continuous map $Y \rightarrow X$,
then $X$ is $G$-Noetherian.
\end{lm}

\begin{proof} 
The pre-image of a descending chain of $G$-stable closed
subsets of $X$ is such a chain in $Y$, hence stabilises. By
surjectivity the chain in $X$ also stabilises.
\end{proof}

\begin{lm} \label{lm:Union}
Let $X$ and $Y$ be $G$-spaces. Then their disjoint union, equipped
with the disjoint-union topology, is $G$-Noetherian if and
only if both $X$ and $Y$ are.
\end{lm}

\begin{proof}
A closed $G$-stable subset of $X \cup Y$ is of the form $C
\cup D$ with $C$ and $D$ closed and $G$-stable in $X$ and
$Y$, respectively. A descending chain of such sets
stabilises if and only if the induced chains in $X$ and $Y$
stabilise. 
\end{proof}

The above lemmas are exact analogues of statements on ordinary
Noetherianity of topological spaces. Now, however, we introduce a
construction where the $G$-structure plays an essential role. Suppose
that $Y$ is an $H$-space, where $H$ is a subgroup of $G$. Construct
the space $G \times_H Y:=G \times Y/\sim$ where $\sim$ is the smallest
equivalence relation with $(gh,y) \sim (g,hy)$ for all $g\in G,h \in
H,y \in Y$. The space $G \times Y$ carries the product topology with $G$
discrete, so that every closed subset is of the form $\bigcup_{g \in G}
\{g\} \times Y_g$ with each $Y_g$ closed in $G$. The space $G \times_H Y$
caries the quotient topology of the product topology and a
topological $G$-action by left multiplication.

\begin{lm} \label{lm:FibreBundle}
If $Y$ is $H$-Noetherian, then $G \times_H Y$ is
$G$-Noetherian.
\end{lm}

\begin{proof}
Let $Z_1 \supseteq Z_2 \supseteq \ldots$ be a chain of $G$-stable closed
subsets of $G \times_H Y$. The pre-image of $Z_i$ in $G \times Y$ is
of the form $\bigcup_{g \in G} \{g\} \times Y_{i,g}$ with $Y_{i,g}$
closed in $Y$. As $Z_i$ is $G$-stable, we have $Y_{i,g}=Y_{i,e}$ for
all $g \in G$, and as $Z_i$ is a union of equivalence classes $Y_{i,e}$
is $H$-stable.  The chain $Y_{1,e} \supseteq Y_{2,e} \supseteq \ldots$
stabilises as $Y$ is $H$-Noetherian, hence so does the
chain $Z_1 \supseteq Z_2 \supseteq \ldots$. 
\end{proof}

In our application to the $k$-factor model, the topological spaces will
be sets of rational points of affine schemes over $K$, equipped with
the Zariski topology where closed sets are given by the vanishing by of
elements in the corresponding $K$-algebra. The following lemma describes
what $G$-Noetherianity means in this case.

\begin{lm} \label{lm:GNoethForKSets}
Suppose that $A=\Spec T$ is an affine scheme over $K$, where $T$ is
a $K$-algebra, and that $G$ acts by automorphisms on $A$, hence on
$T$. Then the following two statements are equivalent:
\begin{enumerate}
\item $A(K)$, equipped with the Zariski topology, is
$G$-Noetherian; and 
\item for every $G$-stable ideal $I$ of $T$ there
exist finitely many elements $f_1,\ldots,f_l$ such that $y
\in A(K)$ lies in the closed set defined by $I$ if and only
if $f_i(gy)=0$ for all $i=1,\ldots,l$ and for all $g \in G$.
\end{enumerate}
\end{lm}

\begin{proof}
Suppose first that $A(K)$ is $G$-Noetherian and let $I$ be a $G$-stable
ideal in $T$. Construct a chain $X_0 \supseteq X_1 \supseteq \ldots$ of
$G$-stable closed subsets of $A(K)$ as follows: $X_0:=A(K)$, and for $i
\geq 1$ either choose $f_i \in I$ which does not vanish identically on
$X_{i-1}$ and set $X_i:=\{y \in X_{i-1} \mid f_i(gy)=0 \text{ for all }
g \in G\}$ or, if such an $f_i$ does not exist, then $X_i=X_{i-1}$. By
$G$-Noetherianity this $G$-stable chain stabilises at some $X_l$, and
then $f_1,\ldots,f_l$ have the required property.

For the converse suppose that every $G$-stable ideal $I$ has the stated
property, and consider a chain $X_1 \supseteq X_2 \supseteq
\ldots$ of $G$-stable closed subsets. Let $I_n$ be the ideal
in $T$ vanishing on $X_n$. The union of all $I_n$ is a
$G$-stable ideal in $T$, hence let $f_1,\ldots,f_l$ be as in
the assumption. Let $n$ be such that $f_1,\ldots,f_l \in
I_n$. Then we claim that $X_m=X_n$ for $m \geq n$. Indeed,
if not, then let $y \in X_n \setminus X_m$. This means that
some element of $I_m \subseteq I$ does not vanish on $y$,
which contradicts the fact that $f_i(gy)=0$ for all $i$ and
$g$. 
\end{proof}

\section{Scheme-theoretic $G$-Noetherianity}
\label{sec:SchemeGNoetherianity}

In this section we introduce and develop the notion of $G$-Noetherianity
for rings. Dually, we will also adopt this terminology for
the corresponding affine schemes. 

\begin{de}
Let $G$ be a group acting by automorphisms on a ring $R$; we shall simply
call $R$ a $G$-ring. Then $R$ is called {\em $G$-Noetherian} if
every chain $I_1 \subseteq I_2 \subseteq \ldots$ of $G$-stable ideals
stabilises. 
\end{de}

For all main theorems we need the following fundamental result.

\begin{thm}[\cite{Aschenbrenner07,Hillar08}] \label{thm:AHS}
The ring $K[x_{ij} \mid i=1,\ldots,l,\ j=0,1,2,\ldots]$, on which the
group $\Sym(\NN)$ of bijections from $\NN$ to itself acts by $\sigma
x_{ij}=x_{i\sigma j}$, is $\Sym(\NN)$-Noetherian.
\end{thm}

This theorem was first proved in \cite{Aschenbrenner07} for the case where
$l=1$, and then generalised in \cite{Hillar08}. Its proof boils down to
showing that a certain order on monomials is a well-quasi-order. The fact
that $\Sym(\NN)$-stable {\em monomial} ideals are finitely generated up to the
action of $\Sym(\NN)$ boils down to the statement that Young diagrams
are well-quasi-ordered by inclusion. This, in turn, is a special case
of the theorem in \cite{MacLagan01} that antichains of monomial ideals
are finite.

\begin{re}
In view of Section \ref{sec:FinitenessProblems} it is more
natural to replace $\Sym(\NN)$ by the direct limit $G$ of all
$\Sym(n)$, where $\Sym(n)$ is considered as the stabiliser
in $\Sym(n+1)$ of $n+1$. As both groups have the same orbits
on the ring above, that ring is also $G$-Noetherian.
\end{re}

\begin{lm} \label{lm:BasisThm}
If $R$ is a $G$-Noetherian ring, then so is $R[X]$, where
$X$ is a variable and $G$ acts only on the coefficients of
the polynomials in $R[X]$.
\end{lm}

\begin{proof}
One can copy the proof of Hilbert's basis theorem from
\cite{Lang65} word-by-word.
\end{proof}

\begin{lm} \label{lm:SchemeImpliesSetNoeth}
Let $R$ be a $K$-algebra, and let $A=\Spec R$ be the corresponding affine
scheme over $K$. Suppose that a group $G$ acts on $R$ by $K$-algebra
automorphisms. Then if $R$ is $G$-Noetherian, then the set $A(K)$ with the
Zariski topology is a $G$-Noetherian topological space. 
\end{lm}

\begin{proof}
Consider a chain $X_1 \supseteq X_2 \ldots$ of $G$-stable closed
subsets of $Y(K)$. Let $I_n$ be the vanishing ideal in $R$
of $X_n$. The $I_n$ form an ascending chain of $G$-stable
ideals, which stabilises as $R$ is $G$-Noetherian. Hence,
since $X_n$ is the zero set of $I_n$, the chain $X_1
\supseteq X_2 \supseteq \ldots$ also stabilises.
\end{proof}

We collect some further elementary properties of scheme-theoretic
$G$-Noetherianity. First, an analogue of Lemma
\ref{lm:ClosedSubset}.

\begin{lm} \label{lm:Epimorphism}
Let $R$ and $S$ be $G$-rings. If there exists a $G$-equivariant
epimorphism $R \rightarrow S$ and $R$ is $G$-Noetherian, then $S$
is $G$-Noetherian.
\end{lm}

\begin{proof}
Every chain $I_1 \subseteq I_2 \subseteq \ldots$ of $G$-stable ideals
in $S$ lifts to a chain in $R$. As the latter stabilises by
$G$-Noetherianity of $R$, so does the chain in $S$ by
surjectivity of the morphism $R \to S$.
\end{proof}

An analogue of Lemma \ref{lm:Union} is this.

\begin{lm} \label{lm:DirectSum}
Let $R$ and $S$ be $G$-rings. Then $R \oplus S$ is $G$-Noetherian
(respectively, radically $G$-Noetherian) if and only if both $R$ and $S$
are $G$-Noetherian (respectively, radically $G$-Noetherian).
\end{lm}

\begin{proof}
A $G$-stable ideal of $R \oplus S$ is of the form $I \oplus J$ with $I$
a $G$-stable ideal in $R$ and $J$ a $G$-stable ideal in $S$. A chain of
such ideals stabilises if and only if the two component chains stabilise.
\end{proof}

Here is one possible analogue of Lemma \ref{lm:Surjection}.

\begin{lm} \label{lm:Monomorphism}
Let $R$ and $S$ be $G$-rings. If there exists a $G$-equivariant
homomorphism $\phi:R \rightarrow S$ such that $\phi^{-1}(\phi(I)S)=I$
for all $G$-stable ideals $I$ of $R$, and if $S$ is $G$-Noetherian,
then so is $R$.
\end{lm}

\begin{proof}
For a chain $I_1 \subseteq I_2 \subseteq \ldots$ of $G$-stable ideals
in $R$ the ideals $J_i:=\phi(I_i)S$ form a chain of $G$-stable ideals in
$S$. As $S$ is $G$-Noetherian, we have $J_{i+1}=J_i$ for all sufficiently
large $i$. But then also $I_{i+1}=\phi^{-1}(J_{i+1})=\phi^{-1}(J_i)=I_i$,
as required.
\end{proof}

For scheme-theoretic finiteness of chirality varieties in characteristic
$0$ we need another construction of $G$-Noetherian algebras. Consider a
$K$-algebra $R$ acted upon by two groups $G$ and $H$, where the actions
have the following four properties: $G$ and $H$ act by $K$-algebra
automorphisms; actions of $G$ and $H$ on $R$ commute; every element
of $R$ is contained in a finite-dimensional $H$-module; and every
finite-dimensional $H$-submodule of $R$ splits as a direct sum of
irreducible $H$-modules. By the second property, the $K$-algebra $R^H$
of $H$-invariants is $G$-stable.

\begin{prop} \label{prop:Reynolds}
If $R$ is $G$-Noetherian, then so is $R^H$. 
\end{prop}

The proof of this proposition uses the Reynolds operator $\rho:R
\rightarrow R^H$, defined as follows. For $f \in R$ let $U$ be a
finite-dimensional $H$-submodule of $R$ containing $f$. Split $U=U_0
\oplus U_1$ where $U_0$ is the sum of all trivial $H$-modules in $U$
and $U_1$ is the sum of all non-trivial irreducible $H$-modules in
$U$. Split $f=f_0+f_1$ accordingly. Then $\rho(f):=f_0$. A standard
verification shows that this map is well-defined and an $R^H$-module
homomorphism $R \rightarrow R^H$. See, for instance, 
\cite{Derksen02,Goodman98,Kraft96}.

\begin{proof}
By Lemma \ref{lm:Monomorphism} it suffices to show that $R I \cap
R^H=I$ for all ideals $I$ of $R^H$. This follows from a standard argument
involving the Reynolds operator: Write $f \in R I \cap R^H$ as $\sum_i
r_i f_i$ with $r_i \in R$ and $f_i \in I$. As $f$ is
$H$-invariant we have 
\[ f=\rho(f)=\sum_i \rho(r_i) f_i \in I, \]
where the last step uses that $\rho$ is an $R^H$-module
homomorphism.
\end{proof}

\section{Proofs of the main theorems}

We retain the setting and notation of Section
\ref{sec:FinitenessProblems}. If we can prove that the ambient
topological space $A_\infty(K)$ is $G$-Noetherian, then by
Lemma \ref{lm:GNoethForKSets} there exist finitely many elements
$f_1,\ldots,f_l$ of the ideal $I_\infty$ of $Y_\infty$ such that $y \in
A_\infty(K)$ lies in $Y_\infty(K)$ if and only if $f_i(gy)=0$ for all
$i$ and all $g \in G$. Choosing $m$ such that $f_1,\ldots,f_l \in I_m$
we then have, for $n \geq m$,
\[ Y_n(K)=\{y \in A_n(K) \mid f_i(gy)=0 \text{ for }
	i=1,\ldots,l \text{ and all } g \in G_n\} \]
by Lemma \ref{lm:InftyChain}, which proves the desired set-theoretic
result. As similar reasoning, assuming that $A_\infty$ is
scheme-theoretically $G$-Noetherian, would yield that for some $m$ and
all $n \geq m$ the ideal of $Y_n$ is generated by the $G_n$-translates
of the ideal of $Y_m$.  Unfortunately, neither the topological space
$\OMn[\infty](K)$ nor the scheme $\Snk{\infty}$ is $G$-Noetherian,
as the following example shows.

\begin{ex}
Consider the monomials
\[ f_2:=y_{12} y_{21},\ f_3:=y_{12}y_{23}y_{31},\
f_4:=y_{12}y_{23}y_{34}y_{41},\ \ldots \]
in the coordinate ring of $\OMn[\infty]$, as well as the
points $p_2,p_3,\ldots \in \OMn[\infty](K)$ where $p_i$ is
an off-diagonal $\NN \times \NN$-matrix with $1$'s on the positions
corresponding to the variables appearing in $f_i$ and
zeroes elswhere. Then we have $f_i(Gp_j)=0$ for all $i \neq
j$ and $f_i(p_i)=1$. Hence the sequence $X_1 \supseteq X_2
\supseteq \ldots$ of $G$-stable closed sets defined by 
\[ X_i:=\{p \in \OMn[\infty](K) \mid f_j(Gp)=\{0\} \text{
for all } j \leq i\} \]
does not stabilise, since $p_{i+1} \in X_i \setminus X_{i+1}$. It
is easy to find a similar example showing that $\Snk{\infty}$ is not
$G$-Noetherian; see \cite[Proposition 5.2]{Aschenbrenner07}.
\end{ex}

Our strategy in both cases is to replace $A_\infty$ by a closed $G$-stable
subscheme $\tA$, which contains $Y_\infty$ and such that $\tA$ {\em
is} $G$-Noetherian (for chirality varieties) or at least $\tA(K)$ is
$G$-Noetherian (for the $k$-factor model).

\subsection*{The $k$-factor model}
In this section $\tA$ equals the subscheme $\tOMk$ of $\OMn[\infty]$
whose ideal is generated by all off-diagonal $(k+1)\times(k+1)$-minors
of the off-diagonal matrix $(y_{ij})_{i \neq j}$. 

\begin{thm}
The topological $\Sym(\NN)$-space $\tOMk(K)$ is $\Sym(\NN)$-Noetherian. 
\end{thm}

\begin{proof}
We proceed by induction on $k$. For $k=0$ the statement is trivial,
since $\tOMk(K)$ consists of a single point. Suppose that the
statement is true for $k-1$. We shall construct a continuous and
$\Sym(\NN)$-equivariant map $\phi$ from a $\Sym(\NN)$-Noetherian
space to $\OMn[\infty](K)$ whose image contains $\tOMk(K)$ as a closed
subset. By Lemmas \ref{lm:Surjection} and \ref{lm:ClosedSubset} we are
then done. The required $\Sym(\NN)$-Noetherian space is the disjoint
union of $\tOMk[k-1](K)$, on which $\phi$ is the inclusion map, and a
second space $Z$, which will cover all points of $\tOMk(K)$ that are
not in $\tOMk[k-1](K)$.

For any (possibly infinite) matrix $Q$ and subsets $L,N$ of its row
index set and column index set, respectively, we write $Q[L,N]$ for the
corresponding submatrix of $Q$. To motivate the construction of $Z$,
set $I:=\{1,\ldots,k\}$ and $J:=\{k+1,\ldots,2k\}$ and consider a point
$Y$ in $\tOMk(K)$ such that $\det Y[I,J]$ is non-zero. We argue that
the matrix $Y[\NN \setminus J, \NN \setminus I]$ is an honest rank-$k$
matrix in the sense that there exist $b_{ip},c_{pj},\ i \in \NN \setminus
J, j \in \NN \setminus I, p=1,\ldots,k$ such that for all such $i,j$
with $i \neq j$ we have
\[ y_{ij}=\sum_{p=1}^k b_{ip} c_{pj}. \]
Indeed, it is clear that we can choose the $b_{ip}$ and $c_{pj}$ such that
this relation is satisfied for $(i,j) \in I \times (\NN \setminus I) \cup
(\NN \setminus J)\times J$. Then for $(i,j) \in (\NN \setminus (I \cup J))
\times (\NN \setminus (I \cup J))$ the relation will automatically be
satisfied due to the vanishing of the determinant of $Y[I \cup \{i\},J
\cup \{j\}]$ and the non-vanishing of the determinant of $Y[I,J]$. We
shall think of the remaining entries of $Y$, i.e., those $y_{ij}$ with
$i \in J$ or $j \in I$, as ``free variables''.  This leads us to consider
the ring
\[ 
R:=K[
(b_{ip})_{i \in \NN \setminus J, 1 \leq p \leq k},
(c_{pj})_{j \in \NN \setminus I, 1 \leq p \leq k},
(d_{ij})_{i \in \NN, j \in I},
(e_{ij})_{i \in J, j \in \NN \setminus I} 
].
\] 
On this ring the group $H:=\Sym(\NN \setminus (I \cup J))$, considered
as the pointwise stabiliser of $I \cup J$ in $\Sym(\NN)$, acts by
permuting the row indices of $b$ and $d$ and the column indices of $c$
and $e$. By Theorem \ref{thm:AHS} and Lemma \ref{lm:BasisThm} the ring
$R$ is $H$-Noetherian, since apart from $4k$ copies of countably many
variables on which the full symmetric group acts, $R$ has only finitely
many further variables. Hence by Lemma \ref{lm:SchemeImpliesSetNoeth}
the topological space $X:=(\Spec R)(K)$ is also $H$-Noetherian. Consider
the map $\phi_X:X \rightarrow \OMn[\infty](K)$ sending $(B,C,D,E)$ to the
off-diagonal matrix
\[
\begin{bmatrix}
D[I,I] & (B.C)[I,J] & (B.C)[I,\NN \setminus (I \cup J)]\\
D[J,I] & E[J,J] & E[J,\NN \setminus (I \cup J)]\\
D[\NN \setminus (I \cup J),I] & (B.C)[\NN \setminus (I \cup J),
J] & (B.C)[\NN \setminus (I \cup J),\NN \setminus (I \cup
J)]
\end{bmatrix},
\]
where the blocks on the diagonal
are projected into the relevant spaces of off-diagonal matrices.
The map $\phi_X$ is continuous and $H$-equivariant, and hence gives
rise to a unique continuous and $\Sym(\NN)$-equivariant map $\phi_Z$
from $\Sym(\NN) \times_H X$ into $\OMn[\infty](K)$ which maps the
equivalence class of $(e,x)$ to $\phi_X(x)$ for all $x$. By Lemma
\ref{lm:FibreBundle} the space $Z$ is $\Sym(\NN)$-Noetherian. As all
off-diagonal $k \times k$-minors are in the same $\Sym(\NN)$-orbit
(up to a sign), the above discussion shows that the image of $Z$
contains $\tOMk(K) \setminus \tOMk[k-1](K)$. Now the disjoint union of
$\tOMk[k-1]$ and $Z$ is $\Sym(\NN)$-Noetherian by Lemma \ref{lm:Union},
and the map $\phi$ which is the inclusion on $\tOMk[k-1]$ and $\phi_Z$
on $Z$ has $\im(\phi) \supset \tOMk$. This proves the theorem.
\end{proof} 

\begin{proof}[Proof of set-theoretic finiteness for the
$k$-factor model]

We spell out the proof of the first statement, which characterises
$\OMnk{n}$ for large $n$ by the condition that all principal $N_0
\times N_0$-submatrices lie in $\OMnk{N_0}$. First we argue that
there exist finitely many elements $f_1,\ldots,f_l$ in the
ideal of $\OMnk{\infty}$ such that $y \in \OMn[\infty](K)$ lies in
$\OMnk{\infty}(K)$ if and only if $f_i(gy)=0$ for $i=1,\ldots,l$
and all $g \in \Sym(\NN)$. Take $f_1$ equal to any off-diagonal
$(k+1)\times(k+1)$-determinant; these form a single
$\Sym(\NN)$-orbit up to a sign.
Requiring that $f_1(gy)=0$ for all $g$ forces $y$ to lie
in $\tOMk(K)$. As the latter space is $\Sym(\NN)$-Noetherian,
the closed $\Sym(\NN)$-stable subspace $\OMnk{\infty}(K)$ is cut
out by finitely many further equations $f_2,\ldots,f_l$; see Lemma
\ref{lm:GNoethForKSets}. Now take $N_0$ large enough such that $f_1,\ldots,f_l$ lie in
the coordinate ring of $\OMn[N_0]$. Then we have $y \in \OMnk{\infty}(K)$
if and only if $\pi_{\infty,N_0}(gy) \in \OMnk{N_0}(K)$ for all $g \in
\Sym(\NN)$. By Lemma \ref{lm:InftyChain} this implies that for all $n
\geq N_0$ an element $y \in \OMn(K)$ lies in $\OMnk{n}(K)$ if and only
if $\pi_{n,N_0}gy$ lies in $\OMnk{N_0}$ for all $g \in \Sym(n)$. This
proves the first statement of the theorem.

The second statement, which concerns the Zariski closure of
the $k$-factor model $F_{n,k}$, is proved in a similar fashion:
$\SOMnk{\infty}(K)$ is a closed $\Sym(\NN)$-stable subspace of $\tOMk(K)$,
hence characterised by finitely many equations. 
\end{proof}

\subsection*{Chirality varieties}

For chirality varieties we take $\tA$ to be the subscheme $\tVk$
of $\Snk{\infty}$ defined by all Pl\"ucker relations among the $y_J$
with $|J|=k$. From the parameterisation
\[ y_J=\prod_{i,j \in J, i<j} (x_i-x_j)=\det(x_j^i)_{j \in
J,\ i=0,\ldots,k-1} \]
it is clear that $\Vnk{\infty}$ is a subscheme of $\tVk$.  Theorem
\ref{thm:VandermondeScheme} will follow from the following theorem.

\begin{thm} \label{thm:tVkNoeth}
Assume that the characteristic of $K$ is zero. Then $\tVk$ is
scheme-theoretically $\Sym(\NN)$-Noetherian.
\end{thm}

\begin{proof}
Consider the scheme $X=\Mmn{k,\NN}$ of $k \times \NN$-matrices with
coordinate ring $R=K[x_{ij} \mid i \in [k], j \in \NN]$.  Let $G =
\Sym(\NN)$ act on $X$ by permuting the columns, and let $H = \SL_k(K)$
act on $X$ by multiplication from the left. Now the conditions of
Proposition \ref{prop:Reynolds} are satisfied: complete reducibility for
$H$ in characteristic $0$ is classical, and $G$-Noetherianity of $R$ is
Theorem \ref{thm:AHS}. Hence $R^H$ is $G$-Noetherian. Now we claim that
the homomorphism sending $y_J$ to $\det(x[[k],J])$ is a
$\Sym(\NN)$-equivariant isomorphism
from the coordinate ring of $\tVk$ to $R^H$. This claim follows
from two well-known facts: First, the kernel of this homomorphism is
generated by the Pl\"ucker relations, which generate the defining ideal
of $\tVk$. Second, by the First Fundamental Theorem for $\SL_k$ the ring
of $\SL_k$-invariants on any space $\Mn[k,n]$ of {\em finite} matrices
is generated by the determinants $\det(x[[k],J])$ with $J \subseteq
[n]$ of size $k$ see \cite{Derksen02,Goodman98,Kraft96,Weyl39}; this
readily implies that $R^H$ is generated by these determinants as $J$
runs through all $k$-sets in $\NN$. Hence $\Vnk{\infty}=\Spec R^H$
is $Sym(\NN)$-Noetherian, as claimed.
\end{proof}

\begin{proof}[Proof of scheme-theoretic finiteness of
chirality varieties in characteristic zero]
The scheme $\tVk$ is cut out scheme-theoretically from $\Snk{\infty}$
by the Pl\"ucker relations, which form a single $\Sym(\NN)$-orbit. By
Theorem \ref{thm:tVkNoeth} the scheme $\tVk$ is $\Sym(\NN)$-Noetherian,
hence its subscheme $\Vnk{\infty}$ is cut out scheme-theoretically
from $\tVk$ by finitely many $\Sym(\NN)$-orbits of equations. Hence
the ideal of $\Vnk{\infty}$ in the coordinate ring of $\Snk{\infty}$ is
generated by finitely many $\Sym(\NN)$-orbits of equations, and Lemma
\ref{lm:InftyChain}, together with the remark following it, concludes
the proof.
\end{proof}

\section{Remarks}

We conclude the paper with a few remarks.
\begin{enumerate}

\item If one drops the characteristic-$0$ assumption in Theorem
\ref{thm:VandermondeScheme} one can still prove a set-theoretic finiteness
result: $\tVk(K)$ is a $\Sym(\NN)$-noetherian topological space.

\item In Theorem \ref{thm:VandermondeScheme} one may replace the
Vandermonde determinant by any other determinant of a square matrix
of which the entry at position $(i,j)$ equals $p_i(x_j)$ for some
fixed polynomials $p_1,\ldots,p_k$. Indeed, the resulting scheme
is still a closed subscheme of $\tVk$, and the argument in Section
\ref{sec:FinitenessProblems} putting the chirality varieties into the
framework of Lemma \ref{lm:InftyChain} applies unaltered.

\item So far we have not succeeded to prove scheme-theoretic finiteness
for the $k$-factor model. Even in the case where $k=1$, in which all
off-diagonal $2 \times 2$-determinants of an infinite off-diagonal matrix
are known to generate the ideal, it is not obvious that the quotient by
these determinants is $\Sym(\NN)$-Noetherian.

\item Proposition \ref{prop:Reynolds} is a powerful tool in proving
non-trivial finiteness results. We expect that it will be useful in many
other problems, as well.
\end{enumerate}


\end{document}